\documentclass{siamart190516}
\usepackage{amsmath,amssymb}
\usepackage{graphicx}
\usepackage{amsfonts}
\usepackage{float}
\usepackage{color}
\usepackage{algorithm}
\usepackage{algorithmic}
\usepackage{hyperref}
\usepackage{savesym}
\usepackage{datetime}
\usepackage[normalem]{ulem}
\usepackage{multirow}
\usepackage{framed} 
\usepackage{xcolor} 
\colorlet{shadecolor}{gray!20} 
\definecolor{mygreen}{rgb}{0,0.6,0}

\newtheorem{remark}[theorem]{Remark}

\def\e{\mathrm{e}}

\def\e{{\epsilon}}

\def\tr{\mathrm{tr}}

\def\Rbnk{{\mathbb{R}^{2n\times 2k}}}
\def\Rbnn{{\mathbb{R}^{2n\times 2n}}}
\def\Spkn{{\mathrm{Sp}(2k,2n)}}

\def\Spn{{\mathrm{Sp}(2n)}}

\def\SPSD{\mathcal{SPSD}}
\def\diag{\mathrm{diag}}

\def\calJ{\mathcal{J}}
\def\calA{\mathcal{A}}
\def\calX{\mathcal{X}}

\def\e{\mathrm{e}}
\def\R{\mathbb{R}}
\def\spn{\mathrm{span}}

\hypersetup{colorlinks=true,linkcolor=blue,citecolor=blue}    

\title{Symplectic eigenvalues of positive-semidefinite matrices and the trace minimization theorem
}

\author{Nguyen Thanh Son\thanks{Institute of Mathematics $\&$ Centre for Advanced Analytics and Predictive Sciences (CAAPS), University of Augsburg, Universit\"{a}tsstra\ss e 12a, 86159 Augsburg, Germany and Thai Nguyen University of Sciences, 24118 Thai Nguyen, Vietnam (ntson@tnus.edu.vn).}
	\and Tatjana Stykel\thanks{Institute of Mathematics $\&$ Centre for Advanced Analytics and Predictive Sciences (CAAPS), University of Augsburg, Universit\"{a}tsstra\ss e 12a, 86159 Augsburg,  Germany (tatjana.stykel@math.uni-augsburg.de).}	
}

\begin{document}
\maketitle

\begin{abstract} 	
Symplectic eigenvalues are conventionally defined for 
symmetric positive-definite matrices via Williamson's diagonal form. Many properties of standard eigenvalues, including the trace minimization theorem, are extended to the case of symplectic eigenvalues. In this note, we will ge\-ne\-ralize Williamson's diagonal form for symmetric positive-definite matrices to the case of symmetric positive-semidefinite matrices, which allows us to define symplectic eigenvalues, and prove the trace minimization theorem in the new setting.
\end{abstract}



\pagestyle{myheadings}
\thispagestyle{plain}
\markboth{N. T. SON AND T. STYKEL}{SYMPLECTIC EIGENVALUES OF SPSD MATRICES AND TRACE MINIMIZATION}

\section{Introduction}
Let us first recall that a matrix $S \in \Rbnn$ is referred to as a \emph{symplectic matrix} if it satisfies the relation
$$
S^TJ_{2n}S = J_{2n}\quad \mbox{with}\quad J_{2n} = \begin{bmatrix}
0&I_n\\-I_n&0
\end{bmatrix},
$$
where $I_n$ denotes the $n\times n$ identity matrix. The set of such matrices forms a multiplicative group \cite{Fiori2011} and is denoted by $\Spn$. We also employ the set of rectangular symplectic matrices defined as
$$
\Spkn = \{S\in \R^{2n\times 2k}\,:\, S^TJ_{2n}S = J_{2k}\}
$$
for some $k$ with $1\leq k \leq n$. This set was shown to be an embedded Riemannian manifold of $\R^{2n\times 2k}$ \cite[Proposition 3.1]{GaoSAS2021}.

Derived from Williamson's work \cite{Will36}, for a $2n\times 2n$ symmetric positive-definite (spd) real matrix $A$, there exists a matrix $S\in \Spn$ such that
\begin{equation}\label{eq:sympleigdecomp}
	S^TAS = \begin{bmatrix}
		D & 0 \\ 0 & D
	\end{bmatrix},
\end{equation}
where $D=\diag(d_1,\ldots,d_n)$ is a diagonal matrix with positive diagonal entries. The equality \eqref{eq:sympleigdecomp} looks very like the eigenvalue decomposition of spd matrices except for the fact that the transformation matrix $S$ is symplectic instead of orthogonal. The right-hand side of \eqref{eq:sympleigdecomp} is termed as \emph{Williamson's diagonal form} of $A$. The positive numbers $d_1,\ldots,d_n$ are referred to as \emph{symplectic eigenvalues}. They form  the \emph{symplectic spectrum} of $A$. The symplectic spectrum of an spd matrix is unique while the so-called \emph{diago\-na\-lizing matrix} $S$ in \eqref{eq:sympleigdecomp} is not unique. The set of diagonalizing matrices for a given spd matrix was characterized via 
a symplectic version of Autonne's uniqueness theorem in \cite[Theorem 3.5]{SonAGS2021}. %

Symplectic eigenvalues find applications in quantum mechanics, optics, and sta\-bi\-lity analysis of structured mechanical systems \cite{Hiro06,BennFS08,KrbeTV14}. Symplectic eigenvalues and eigenvectors, defined below,  can be numerically computed 
using a symplectic Lanczos method via the connection with  so-called positive-definite Hamiltonian matrices \cite{Amod06} or by solving a trace minimization problem using a Riemannian optimization method \cite{GaoSAS2021,SonAGS2021}. 

Many properties of symplectic eigenvalues for spd matrices have been investigated. Besides trace minimization mentioned above, these are Cauchy's interlacing theorem \cite{BhatJ15}, variational principles and Weyl inequality \cite{BhatJ20}, {S}chur-{H}orn theorem \cite{BhatJ20_2}, {S}zeg\H{o}-type theorem \cite{BhatJ21}, and Lidskii's (or the majorization) theorem \cite{JainM22}, just to name a few. 

It is noteworthy that in the original work \cite{Will36}, $A$ is not necessarily spd but solely symmetric. It would be a reasonable question to ask if it is still possible to define symplectic eigenvalues and investigate their properties for a more general case in which $A$ is symmetric positive-semidefinite (spsd). It turns out that not every spsd matrix is appropriate for this purpose and even in the case that symplectic eigenvalues are definable, 
some properties that are already known for symplectic eigenvalues of spd matrices are nontrivially extendable to the new setting. In this note, we will set the condition under which an spsd matrix enjoys Williamson's diagonal form which enables the definition of symplectic eigenvalues. Moreover, we also prove the trace minimization theorem for symplectic eigenvalues of such matrices because it can not only be exploited to compute symplectic eigenpairs \cite{SonAGS2021} but it also provides a direct way to derive the majorization inequality similar to \cite[Theorem 1]{Hiro06}.

For ease of presentation, we now collect some necessary materials following \cite{deGo06,JainM22,SonAGS2021}. The Kronecker delta function will be denoted and understood as 
$$\delta_{ij} = \begin{cases}
1,\ \mbox{ if } i = j,\\
0,\ \mbox{ if } i \not= j.
\end{cases}
$$
The superscript $\ \cdot^T$ applied to a matrix means its transpose. The set of all $2n\times 2n$ spsd real matrices will be denoted by  $\SPSD(2n)$. Given a square matrix $M$, $\tr (M)$ denotes the trace of $M$; the \emph{kernel} or the \emph{null space} of $M$ is the set $\ker M = \{x:Mx = 0\}$; a~\emph{square root} of $M$ is a~square matrix $K$ of the same size such that $K^2 = M$. Furthermore, $\spn\{v_1,\ldots,v_\ell\}$ stands for the subspace 
spanned by the vectors $v_1,\ldots,v_\ell$; applied to a matrix, it means the subspace spanned by the columns of this matrix. Given $A_j~\in~\R^{n_j\times n_j}, j = 1,\ldots,\ell,$ $\diag(A_1,\ldots,A_\ell)$ denotes the (block) diagonal matrix with $A_j$  on the diagonal. A subset $U \subset \R^{2n}$ is said to be \emph{symplectically orthogonal to} a subset $V \subset \R^{2n}$ if 
$u^TJ_{2n}v = 0$ for all $u\in U, v \in V$. 
Two pairs of vectors $(u_1,v_1)$ and $(u_2,v_2)$ in $\R^{2n}\times \R^{2n}$ are said to be \emph{symplectically normalized}\footnote{It is slightly different from the definition in \cite{SonAGS2021}.} if 
\begin{equation*}
u_i^TJ_{2n}v^{}_j = \delta_{ij}\quad \mbox{and}\quad  u_i^TJ_{2n}u^{}_j = v_i^TJ_{2n}v^{}_j = 0 \quad  \mbox{for}\quad i,j = 1,2.
\end{equation*}
The symplecticity of the matrix $[x_1,\ \ldots,\ x_{2k}]$ is equivalent to the fact that each two pairs $(x_i,x_{k+i})$ and $(x_j,x_{k+j})$  are symplectically normalized for $i,j = 1,\ldots,k$ and $i~\not=~j.$ 
It was shown in \cite[Theorem 1.15]{deGo06} that a symplectically normalized set of vectors are linear independent. Therefore, when $k=n$, such vectors form a basis of $\R^{2n}$. In that case, it is called a \emph{symplectic basis}. A subspace $\mathbb{U}\subset \R^{2n}$ is called a \emph{symplectic subspace} if for any $x\in \mathbb{U}\backslash \{0\}$, there exists $y\in \mathbb{U}$ such that $x^TJ_{2n}y = 1$. It turns out that we can always construct a  symplectic basis for a symplectic subspace \cite[Theorem 1.15]{deGo06}. Conversely, a~given subspace is symplectic if it owns a symplectic basis as shown in Lemma~\ref{lem:symplecticbasis} below. Two nonzero vectors written as a matrix $[u,v]\in \R^{2n\times 2}$ form a~\emph{symplectic eigenvector pair} corresponding to the \emph{symplectic eigenvalue} $d$ of $A\in \SPSD(2n)$ if it holds
\begin{equation}\label{Eq:Def_symplEigenVec}
Au = dJ_{2n}v,\quad Av = -dJ_{2n}u.
\end{equation}
The set $([u,v],d)$ is then generally referred to as a \emph{symplectic eigenpair}. The existence of Williamson's diagonal form for any given $2n\times 2n$ spd matrix tells us that there exists a~symplectic basis for $\R^{2n}$ consisting of its symplectic eigenvectors.

\section{Williamson's diagonal form of symmetric positive-semidefinite matrices}
In this section, we will extend the definition of symplectic eigenvalues to the case of spsd matrices by explicitly constructing Williamson's diagonal forms. For further discussion, we will need the following technical property.
\begin{lemma}\label{lem:symplecticbasis}
	If a subspace $\mathbb{W}\subset \R^{2n}$ owns a symplectic basis, then it is a symplectic subspace.
\end{lemma}
\begin{proof}
Assume that the columns of $[w_1,\ldots,w_m,w_{m+1},\ldots,w_{2m}]$  with $1\leq m \leq n$	form a symplectic basis of $\mathbb{W}$. Then any $x \in \mathbb{W}\backslash\{0\}$ can be represented via this basis as $x = \sum_{j=1}^{m}(\lambda_jw_j+\lambda_{m+j}w_{m+j})$  with $\sum_{j=1}^{2m}\lambda_j^2\not= 0$. Setting 
$$y = \frac{1}{\sum_{j=1}^{2m}\lambda_j^2}\sum_{j=1}^{m}(\lambda_jw_{m+j}-\lambda_{m+j}w_j),$$
we can verify that $x^TJ_{2n}y = 1$. This implies that the subspace $\mathbb{W}$ is symplectic.
\end{proof}
For a moment, we consider  an spd matrix $A$ which has Williamson's diagonal form~\eqref{eq:sympleigdecomp}. Multiplying \eqref{eq:sympleigdecomp} with $S^{-T} = J^{}_{2n}SJ_{2n}^T$ from the left and with $S^{-1} = J_{2n}^TS^TJ^{}_{2n}$ from the right, we derive
\begin{equation*}
A = J_{2n}S\begin{bmatrix}
		D & 0 \\ 0 & D
	\end{bmatrix} (J_{2n}S)^T.
\end{equation*}
Then, for $0\leq m < n$, we set
\begin{equation*}
\tilde{A} = J_{2n}S\begin{bmatrix}
		\tilde{D}_{m+1:n} & 0 \\ 0 & \tilde{D}_{m+1:n}
	\end{bmatrix} (J_{2n}S)^T,
\end{equation*}
where $\tilde{D}_{m+1:n} = \diag(0,\ldots,0,d_{m+1},\ldots,d_n)  \in \R^{n\times n}$. We reverse the above multiplications to obtain
\begin{equation*}
\tilde{A}S = J_{2n}S\begin{bmatrix}
0&-\tilde{D}_{m+1:n} \\\tilde{D}_{m+1:n} &0
\end{bmatrix}
\end{equation*}
and 
\begin{equation}\label{eq:sympleigdecomp_modified3}
	S^T\tilde{A}S = \begin{bmatrix}
		\tilde{D}_{m+1:n} & 0 \\ 0 & \tilde{D}_{m+1:n}
	\end{bmatrix}.
\end{equation}
Several facts can be drawn from the above construction. First, $\tilde{A}$ is an spsd matrix of rank $2n-2m$. Second, 
if we write $S$ by its columns as $S = [s_1,\ldots,s_{2n}]$, then the kernel of $\tilde{A}$ is $\spn\{s_1,\ldots,s_m,s_{n+1},\ldots,s_{n+m}\}$. Moreover, it is a symplectic subspace of dimension $2m$ due to Lemma~\ref{lem:symplecticbasis}.
And third, it is symplectically orthogonal to the subspace $\spn\{s_{m+1},\ldots,s_n,s_{n+m+1},\ldots,s_{2n}\}$ which is nothing else but the range of $\tilde{A}$.


Obviously, the decomposition \eqref{eq:sympleigdecomp_modified3} is like \eqref{eq:sympleigdecomp} and therefore it can be thought of as Williamson's diagonal form of the spsd matrix $\tilde{A}$.
This gives us a clear vision of how to construct Williamson's diagonal form of spsd matrices and under which condition this is possible. In practice, we will proceed somewhat in the reverse direction.  Indeed, given $A\in \SPSD(2n)$ whose kernel is a symplectic space. As any symplectic subspace must have even dimension \cite[Sect. 1.1]{deGo06}, we assume that $\dim(\ker A) = 2m$ with $0 \leq m < n$. Let $W = [W_1\ W_2]$ with $W_1,\ W_2 \in \R^{2n\times m}$ be a symplectic matrix whose columns form a symplectic basis of $\ker A$. Note that such a basis can be computed using the symplectic Gram-Schmidt process \cite[Theorem.~ 1.15]{deGo06}. We adapt the procedure proposed in  \cite{Part13} for constructing Williamson's diagonal form of spd matrices to the more general case where the matrix $A$ is spsd 
as follows. 

First, it is worth to recall that any spsd matrix has a unique spsd square root, denoted by $A^{1/2}$, and that $\ker A^{1/2} = \ker A$ \cite[Theorem 7.2.6]{HornJohnson2013}. Next, we will show that $\ker(A^{1/2}J_{2n}A^{1/2}) = \ker A^{1/2}$. It is trivial that $\ker(A^{1/2}J_{2n}A^{1/2}) \supseteq \ker A^{1/2}$. Now, consider any $y\in \R^{2n}$ such that $A^{1/2}J_{2n}A^{1/2}y = 0$. Assume that $J_{2n}A^{1/2}y\not=0$. Because $\ker A^{1/2}$ is a symplectic subspace, there exists $z\in \ker A^{1/2}$ such that $z^TJ_{2n}J_{2n}A^{1/2}y = 1$ which leads to $y^TA^{1/2}z = -1$. This contradicts to the fact that $z\in \ker A^{1/2}$ and hence $J_{2n}A^{1/2}y=0$. As $J_{2n}$ is nonsingular, we deduce that $A^{1/2}y=0$. This implies the desired equality.  
Observe moreover that the matrix $A^{1/2}J_{2n}A^{1/2}$ is skew-symmetric. Then, this matrix can be transformed into the real Schur form
\begin{equation}\label{eq:Schur}
	Q^TA^{1/2}J_{2n}A^{1/2}Q = \diag\left(0_{2\times 2},\ldots,0_{2 \times 2},\begin{bmatrix}
		0&d_{m+1}\\-d_{m+1}&0
	\end{bmatrix},\ldots, \begin{bmatrix}
		0&d_{n}\\-d_{n}&0
	\end{bmatrix}\right),
\end{equation}
where $0 < d_{m+1} \leq\cdots\leq d_{n}$ and $Q$ is a real $2n\times 2n$ orthogonal matrix, see, e.g., \cite[Theorem~7.4.1]{GoluV13}. 
After that, we multiply both sides of \eqref{eq:Schur} with the permutation matrix $P~=~[\e_1,\e_3,\ldots,\e_{2n-1},\e_2,\e_4,\ldots,\e_{2n}]$ from the right and with $P^T$ from the left, where
$\e_j$ denotes the $j$-th canonical unit vector in $\R^{2n}$. 
As a result, we obtain
$$
P^TQ^TA^{1/2}J_{2n}A^{1/2}QP = \begin{bmatrix}
	0&D\\-D& 0
\end{bmatrix},
$$
where $D = \diag(0, \ldots,0,d_{m+1},\ldots,d_{n})$. Finally, setting 
\begin{equation}\label{eq:DiagonalizingMat}
	\tilde{S} = J_{2n}A^{1/2} QP\left[\begin{array}{ll}
		0_{m\times (n-m)}&0_{m\times (n-m)}\\
		0_{(n-m)\times (n-m)}&-\tilde{D}_{m+1:n}^{-1/2}\\
		0_{m\times (n-m)}&0_{m\times (n-m)}\\
		\tilde{D}_{m+1:n}^{-1/2}& 0_{(n-m)\times (n-m)}
	\end{array}\right] =: [\tilde{S}_1\ \tilde{S}_2] \in \R^{2n\times(2n-2m)},
\end{equation}
where $\tilde{D}^{-1/2}_{m+1:n} = \diag(1/\sqrt{d_{m+1}},\ldots,1/\sqrt{d_{n}})$ and $\tilde{S}_1,\tilde{S}_2 \in \R^{2n\times (n-m)}$, 
we define the matrix $S = [W_1\  \tilde{S}_1\ W_2\   \tilde{S}_2]$. 
The following proposition shows that $S$ is a symplectic matrix which diagonalizes $A$. 
\begin{proposition}\label{prop:WillFormSemidefiniteMat}
	Given a matrix $A\in \SPSD(2n)$ 
	whose kernel is a symplectic subspace of $\R^{2n}$ of dimension $2m$. Then the matrix $S$ constructed as above is symplectic and it holds
	\begin{equation}\label{eq:sympleigdecomp_4}
		S^TAS = \begin{bmatrix}
			D & 0 \\ 0 & D
		\end{bmatrix},
	\end{equation}
where $D = \diag(d_1, \ldots,d_{n})$ with $0=d_1 =\cdots = d_m < d_{m+1}\leq \cdots \leq d_{n}$. 
\end{proposition}
\begin{proof}
It follows from $\ker A^{1/2}=\ker A = \spn(W)$ that $A^{1/2}W = 0$. Thus, in view of \eqref{eq:DiagonalizingMat}, $W^TJ_{2n}\tilde{S} =0$. Furthermore, taking into account the fact that the matrices $W$ and $\tilde{S}$ are both symplectic, we deduce that $S$ is symplectic.
	
The diagonal form \eqref{eq:sympleigdecomp_4} follows directly from the equalities $AW_j = 0$,  $W_j^TA = 0$, and $\tilde{S}^T_iA\tilde{S}_j = \delta_{ij}\diag(d_{m+1},\ldots,d_n)$ for $i,j = 1, 2$. 
\end{proof}
In view of the connection of $d_j, j=1,\ldots,n,$ determined in \eqref{eq:sympleigdecomp_4} with the standard eigenvalues of Hamiltonian matrix $J_{2n}A$, see \cite[Proposition 3.2]{SonAGS2021} and references therein, we can conclude that $d_j, j=1,\ldots,n,$ are unique. This fact gives rise to the following definition.
\begin{definition}
For $A\in \SPSD(2n)$ with a symplectic kernel, 
the right-hand side of \eqref{eq:sympleigdecomp_4} is called \emph{Williamson's diagonal form} of $A$. The nonnegative diagonal elements in this form $0 \leq d_1 \leq \cdots  \leq d_n$ are called the \emph{symplectic eigenvalues} of $A$.
\end{definition}
	It immediately follows from \eqref{eq:sympleigdecomp_4} that the symplecticity of the kernel of $A$ is also a necessary condition for the existence of Williamson's diagonal form of $A$. The spsd matrices that do not satisfy this condition are not difficult to find. Consider, for instance, $A = \diag(I_n,0_{n\times n})$. Then $\ker A = \spn\{\e_{n+1},\ldots,\e_{2n}\}$. One can directly verify that $x^TJ_{2n}y = 0$ for all $x, y\in \ker A$, which means that $\ker A$ is an isotropic subspace \cite[Section 1.2]{deGo06}. In fact, all spsd matrices whose kernel is an isotropic subspace do not have  Williamson's diagonal form.
	
	The observation that owning a symplectic kernel is a necessary and sufficient condition for the existence of Williamson's diagonal form of an spsd matrix was previously made in \cite[Remark 2.6]{JainM22}. A justification for this observation can be drawn from \cite[Proposition 5.2]{JainM22} where one of the symplectic projections is uniquely determined by the columns of $W$. Our proof here is however more constructive.

	For the sake of convenience, in the rest of this paper, we restrict ourselves to the case of spsd matrices that satisfy the symplectic kernel condition mentioned in Proposition~\ref{prop:WillFormSemidefiniteMat}. Certainly, $\{0\}$ is itself a (trivial) symplectic subspace. Therefore, the derived results hold for spd matrices too.

\section{Trace minimization theorem}\label{sec:Ky-Fan}
The trace minimization theorem 
for the symplectic eigenvalues of spd matrices was established in \cite{Hiro06,BhatJ15}. To the best of our knowledge, these are the only sources that can be found in the literature. This theorem was exploited in \cite{SonAGS2021} to compute symplectic eigenpairs of spd matrices. A popular way of justifying such a statement is to use mi\-ni\-max principles; see \cite{Sameh82,Nakic03} for standard eigenvalues of Hermitian matrix pencils and \cite{BhatJ15} for symplectic eigenvalues of spd matrices. In \cite{BhatJ15}, the derived mi\-ni\-max principles however invoked the reciprocals of symplectic eigenvalues. This approach is certainly inapplicable to the case of spsd matrices. In this section, we will prove the trace minimization theorem for symplectic eigenvalues in a more general setting where $A$ can be spsd. Note that restricted to the case of spd matrices, our proof differs from the ones presented in \cite{Hiro06,BhatJ15}.  
For completeness, we restate the trace minimization theorem in a more general form in which the symplectic eigenvalues can be zero. 

\begin{theorem}\label{theo: Ky-Fan_symplectic}
Let $A\in \SPSD(2n)$ have symplectic eigenvalues $0\leq d_1 \leq \cdots \leq d_n$. Then for any $k$ with $1\leq k \leq n$, it holds that
\begin{equation*}
\min_{X\in \Spkn}\tr(X^TAX) = 2\sum_{j=1}^kd_j.
\end{equation*}
\end{theorem}

For a proof of this result, we need some preparation. We start with reformulating 
the definition of symplectic eigenvalues and eigenvectors. Recalling \eqref{Eq:Def_symplEigenVec}, 
$[u,v]\in \R^{2n\times 2}$ is a symplectic eigenvector pair corresponding to the symplectic eigenvalue $d$ of $A$ if and only if one of the following four equivalent conditions holds:
\begin{align}
\begin{split}
&\begin{bmatrix}
A&0\\0&A
\end{bmatrix} 
\begin{bmatrix}
u\\v
\end{bmatrix} 
= d 
\begin{bmatrix}
\ \ 0& J_{2n}\\-J_{2n}&0
\end{bmatrix}
\begin{bmatrix}
u\\v
\end{bmatrix},\\
&\begin{bmatrix}
	A&0\\0&A
\end{bmatrix} 
\begin{bmatrix}
	v\\u
\end{bmatrix} 
= -d 
\begin{bmatrix}
\ \	0&J_{2n}\\-J_{2n}&0
\end{bmatrix}
\begin{bmatrix}
	v\\u
\end{bmatrix},\\
&\begin{bmatrix}
A&0\\0&A
\end{bmatrix} 
\begin{bmatrix}
\ \ v\\-u
\end{bmatrix} 
= d 
\begin{bmatrix}
\ \ 0&J_{2n}\\-J_{2n}&0
\end{bmatrix}
\begin{bmatrix}
\ \ v\\-u
\end{bmatrix},\\
&\begin{bmatrix}
A&0\\0&A
\end{bmatrix} 
\begin{bmatrix}
-u\\ \ \ v
\end{bmatrix} 
= -d 
\begin{bmatrix}
\ \ 0&J_{2n}\\-J_{2n}&0
\end{bmatrix}
\begin{bmatrix}
-u\\ \ \ v
\end{bmatrix}.
\end{split}\label{Eq:reform}
\end{align}
This fact suggests constructing the matrix pencil  
\begin{equation}\label{Eq:MatPen}
\calA-\lambda\calJ :=
\begin{bmatrix}
A&0\\0&A
\end{bmatrix}
-\lambda\begin{bmatrix}
\ \ 0&J_{2n}\\-J_{2n}&0
\end{bmatrix}
\end{equation}
with the following properties.

\begin{lemma}\label{lem:PencilProperties}
Let $A\in \SPSD(2n)$ and $\calA -\lambda\calJ$ be the matrix pencil as in \eqref{Eq:MatPen}. Then the following statements hold:
\begin{enumerate}
 \item The real number $d$ is a symplectic eigenvalue of $A$ if and only if $d$ and $-d$ are eigenvalues of multiplicities two of the matrix pencil $\calA -\lambda\calJ$.
\item The matrix $\calA$ is spsd and the matrix $\calJ$ is symmetric and nonsingular.
 \item The matrix $\calJ$ possesses only $1$ and $-1$ as its eigenvalues. Moreover, both of them are of multiplicity $2n$. 
\end{enumerate}
\end{lemma}

\begin{proof}
The first statement follows immediately from \eqref{Eq:reform}  while the second one is straightforward. To justify the last one, let $X_1$ and $X_2$ be any two basis matrices of $\R^n$. Then it is direct to show that the matrices
\begin{equation*}
\begin{bmatrix}
X_1&\ \ 0\\ 0&\ \ X_2\\ 0&-X_2\\X_1&\ \ 0
\end{bmatrix},\quad
\begin{bmatrix}
\ \ X_1&0\\ \ \ 0&X_2\\ \ \ 0&X_2\\ -X_1&0
\end{bmatrix}
\end{equation*}
consist of eigenvectors of $\calJ$ associated with eigenvalues $1$ and $-1$, respectively.
\end{proof}

The relations in \eqref{Eq:reform} together with Lemma~\ref{lem:PencilProperties} establish a connection between  symplectic eigenpairs of $A$ and eigenpairs of the matrix pencil \eqref{Eq:MatPen}. Hence, the result on trace minimization of positive-semidefinite matrix pencils derived in \cite{KovaV95} is here helpful. Let us first collect the necessary facts from there. Given two $\ell \times \ell$ symmetric real matrices $A$ and $B$. 
The matrix pencil $A-\lambda B$ is said to be \emph{positive-semidefinite} if there is a number $\lambda_0 \in \R$ such that $A-\lambda_0 B$ is positive-semidefinite. Such constant $\lambda_0$ is called a \emph{definitizing shift}. Obviously, the matrix pencil $\calA-\lambda\calJ$ in \eqref{Eq:MatPen} is positive-semidefinite with a shift $\lambda_0 = 0$.

Furthermore, it was shown in 
\cite[Proposition 4.1]{KovaV95} that the eigenvalues of the positive-semidefinite matrix pencil $A-\lambda B$ with nonsingluar $B$ are real and the number of non-positive (resp. non-negative) eigenvalues of $A-\lambda B$ is the same as the number of negative (resp. positive) eigenvalues of $B$. In other words, if $q$ and $p$ with $q+p=\ell$ is the number of negative and positive eigenvalues of $B$, respectively, then there are $q$ non-positive eigenvalues of $A-\lambda B$, namely, $\alpha_q^- \leq \cdots \leq \alpha_1^- \leq 0$, and $p$ non-negative eigenvalues of $A-\lambda B$, namely, $0\leq \alpha_1^+\leq \cdots \leq \alpha_p^+$. Thus, they can overall be arranged as
$$
\alpha_q^- \leq \cdots \leq \alpha_1^- \leq \alpha_1^+ \leq \cdots \leq \alpha_p^+.
$$ 
The following theorem provides an important result on trace minimization for positive-semidefinite matrix pencils.
\begin{theorem}\cite[Theorem 3.1]{KovaV95}\label{Theo:KovacVeselic}
Let $A,\, B \in \R^{\ell\times \ell}$ be symmetric and let $B$ be nonsingular with $p$ positive and $q$ negative eigenvalues. Assume moreover that the matrix pencil $A-\lambda B$ is positive-semidefinite. Let $p_1$ and $q_1$ be two integers such that $0\leq p_1 \leq p$ and $0\leq q_1 \leq q$. Then, 
the function
\begin{equation}\label{Eq:KovacVeselicfunc}
\R^{\ell\times (p_1+q_1)}\ni\ X\mapsto \tr (X^TAX) \in \R,
\end{equation}
subjected to the condition
\begin{equation}\label{Eq:KovacVeseliccondition}
X^TBX = \diag(I_{p_1},-I_{q_1}),
\end{equation}
is bounded from below by
\begin{equation}\label{Eq:KovacVeseliclowerbound}
\sum_{j=1}^{p_1}\alpha_j^+ - \sum_{j=1}^{q_1}\alpha_j^-.
\end{equation}
If, additionally, there exists a matrix $X_0\, \in \R^{\ell\times (p_1+q_1)}$ satisfying \eqref{Eq:KovacVeseliccondition} and consisting of eigenvectors of the matrix  pencil $A-\lambda B$ associated with the eigenvalues $\alpha_1^+,\ldots,\alpha_{p_1}^+,\alpha_1^-,\ldots,\alpha_{q_1}^-$, then the lower bound \eqref{Eq:KovacVeseliclowerbound} becomes the minimal value of the function \eqref{Eq:KovacVeselicfunc} which is reached at~$X_0$.
\end{theorem}
Now, we are ready to prove the trace minimization Theorem~\ref{theo: Ky-Fan_symplectic} for symplectic eigenvalues of spsd matrices. 

\noindent\textit{Proof of Theorem \ref{theo: Ky-Fan_symplectic}} 
Consider the matrix pencil $\calA -\lambda\calJ$ as in \eqref{Eq:MatPen}. By Proposition~\ref{prop:WillFormSemidefiniteMat} and Lemma~\ref{lem:PencilProperties}, the eigenvalues of $\calA -\lambda\calJ$, sorted in the nondecreasing order, are
\begin{equation*}
-d_n = -d_n \leq \cdots \leq -d_{1} = -d_{1} \leq d_{1} = d_{1} \leq \cdots \leq d_n = d_n,
\end{equation*}
where $d_ 1,\ldots,d_n$ are the symplectic eigenvalues of $A$. Thus, an application of Theorem~\ref{Theo:KovacVeselic} to the matrix pencil $\calA -\lambda\calJ$ for $p_1 = q_1=2k$ yields
\begin{equation}\label{eq:proof1}
\tr(\calX^T\calA\calX) \geq 4\sum_{j=1}^kd_j
\end{equation}
for all $\calX\in \R^{4n\times 4k}$ satisfying the constraint
\begin{equation}\label{eq:proof2}
\calX^T\calJ\calX = \diag(I_{2k},-I_{2k}).
\end{equation}

Next, for any $S \in \Rbnk$, we construct 
the matrix
\begin{equation}\label{eq:proof3}
\calX = \frac{1}{\sqrt{2}}\left[\begin{array}{ll}
S&SJ_{2k}^T\\
SJ_{2k}^T&S
\end{array}\right] \in \R^{4n\times 4k}.
\end{equation}
This matrix fulfills
\begin{align}\notag
\tr(\calX^T\calA\calX) &= \tr(S^TAS + J_{2k}^{}S^TASJ_{2k}^T)\\ \notag
&=\tr(S^TAS + S^TASJ_{2k}^TJ^{}_{2k})\\
&=2\, \tr(S^TAS).
\label{eq:trace_equal}
\end{align}
Moreover, taking $S\in \Spkn$, one can verify by a direct calculation that the matrix $\calX$ in \eqref{eq:proof3} satisfies \eqref{eq:proof2}. 
Then, using \eqref{eq:trace_equal} and \eqref{eq:proof1}, we obtain that
\begin{align}\label{eq:proof7}
\min_{S \in \Spkn}\tr(S^TAS) 
 &= \frac{1}{2}\min_{\footnotesize{\begin{array}{ll}
&\calX\mbox{ as in \eqref{eq:proof3}},\\ &S\in \Spkn
\end{array}}}\tr(\calX^T\calA\calX)\\ \label{eq:proof8}
 &\geq
\frac{1}{2}\min_{
\tilde{\calX}^T\calJ\tilde{\calX}= \diag(I_{2k},-I_{2k})}
\tr(\tilde{\calX}^T\calA\tilde{\calX})\\ \notag
 & \geq 2\sum_{j=1}^kd_j, 
\end{align}
in which \eqref{eq:proof8} is due to the fact that the constraint set in the right-hand side of \eqref{eq:proof7} is a subset of that in \eqref{eq:proof8}.

Finally, we choose $S = [u_1,\ldots,u_k,v_1,\ldots,v_k]$ consisting of (symplectically) normalized symplectic eigenvector pairs $[u_j,v_j]$ associated with the symplectic eigenvalue $d_j$ of $A$ for $j=1,\ldots,k$. In this case, we have $\tr (S^TAS) = 2\sum_{j=1}^kd_j$
which completes the proof. \hfill $\square$

\begin{remark}
For the computation of the smallest symplectic eigenvalues of the spsd matrix $A$, a direct use of the symplectic Lanczos procedure \cite{Amod06} will not work because it involves the inverse of $A$ which does not exist. 
On the other side, as the theoretical results in \cite{SonAGS2021} still hold for spsd matrices, Theorem~\ref{theo: Ky-Fan_symplectic} plays a key role in enabling the use of the numerical method proposed there for computing the required symplectic eigenpairs.
\end{remark}

\bibliographystyle{siamplain}

\end{document}